\documentclass{amsart}

\usepackage[T1]{fontenc}
\usepackage[utf8]{inputenc}
\usepackage[english]{babel}

\usepackage{booktabs,graphicx,multirow}
\usepackage[all]{xy}
\usepackage[dvipsnames]{xcolor}
\usepackage[colorlinks, linkcolor=BlueViolet, urlcolor=BlueViolet, citecolor=BlueViolet, breaklinks=true]{hyperref}

\theoremstyle{plain}
\newtheorem{theorem}{Theorem}[section]
\newtheorem{proposition}[theorem]{Proposition}
\newtheorem{lemma}[theorem]{Lemma}
\newtheorem{corollary}[theorem]{Corollary}
\newtheorem{algorithm}[theorem]{Algorithm}
\theoremstyle{definition}
\newtheorem{definition}[theorem]{Definition}
\theoremstyle{remark}
\newtheorem{remark}[theorem]{Remark}

\allowdisplaybreaks[1]
\newlength{\algindent}
\newlength{\alglabel}
\newcounter{stepcount}
\newenvironment{algtop}
{\quad\begin{list}{\arabic{stepcount}.}{\leftmargin=\algindent\labelwidth=\algindent\itemsep=\smallskipamount\usecounter{stepcount}}}
{\end{list}}
\newcommand{\algin}{\item[\emph{Input}\textup{:}]}
\newcommand{\algout}{\item[\emph{Output}\textup{:}]}
\newenvironment{alglist}
{\quad\begin{list}{\textup{\arabic{stepcount}}.}{\leftmargin=1.5em\labelwidth=1em\labelsep0.5em\itemsep=\smallskipamount\usecounter{stepcount}}}
{\end{list}}
\newcounter{substepcount}
\newenvironment{algsublist}
{\quad\begin{list}{\textup{(\/\rlap{\alph{substepcount}}\phantom{d}\/)}}{\usecounter{substepcount}}}
{\end{list}}

\newcommand{\mybar}[1]{
  \mathchoice
  {#1\llap{$\overline{\phantom{\displaystyle\rm#1}}$}}
  {#1\llap{$\overline{\phantom{\textstyle\rm#1}}$}}
  {#1\llap{$\overline{\phantom{\scriptstyle\rm#1}}$}}
  {#1\llap{$\overline{\phantom{\scriptscriptstyle\rm#1}}$}}
} 
\renewcommand{\bar}{\mybar}
\renewcommand{\hat}{\widehat}

\newcommand\lowtilde{\lower0.7ex\hbox{\textasciitilde}}

\newcommand{\ie}{\emph{i.e.},~}

\DeclareMathOperator{\Aut}{Aut}
\DeclareMathOperator{\Cl}{Cl}
\DeclareMathOperator{\End}{End}
\DeclareMathOperator{\Gal}{Gal}
\DeclareMathOperator{\GL}{GL}
\DeclareMathOperator{\Id}{Id}
\DeclareMathOperator{\Jac}{Jac}
\DeclareMathOperator{\Norm}{Norm}
\DeclareMathOperator{\Sp}{Sp}

\newcommand{\CC}{\mathbb{C}}
\newcommand{\NN}{\mathbb{N}}
\newcommand{\QQ}{\mathbb{Q}}
\newcommand{\QQbar}{\overline{\mathbb{Q}}}
\newcommand{\RR}{\mathbb{R}}
\newcommand{\ZZ}{\mathbb{Z}}
\newcommand{\xx}{\mathbf{x}}
\newcommand{\yy}{\mathbf{y}}

\newcommand{\Acal}{\mathcal{A}}
\newcommand{\Dcal}{\mathcal{D}}
\newcommand{\Hcal}{\mathcal{H}}
\newcommand{\Ical}{\mathcal{I}}
\newcommand{\Mcal}{\mathcal{M}}
\newcommand{\Ocal}{\mathcal{O}}
\newcommand{\Mb}{\mathbf{M}}
\newcommand{\Aa}{\mathfrak{A}}
\newcommand{\MM}{\mathfrak{M}}
\newcommand{\afrak}{\mathfrak{a}} 

\newcommand{\Adual}{\hat{A}}
\newcommand{\alphadual}{\hat{\alpha}}

\begin{document}

\title[Principally polarized squares of elliptic curves]{Principally polarized squares of elliptic curves\\ with field of moduli equal to $\QQ$}
\date{}

\author[A. G\'elin]{Alexandre G\'elin}
\address{Laboratoire de Math\'ematiques de Versailles,
         UVSQ, CNRS, Universit\'e Paris-Saclay,
         45 avenue des \'Etats-Unis, 
         78035 Versailles, France.}
\email{\href{mailto:alexandre.gelin@uvsq.fr}{alexandre.gelin@uvsq.fr}}
\urladdr{\url{https://alexgelin.github.io/}}
\thanks{This work was supported in part by a public grant as part of the \emph{Investissement d'avenir} project, reference ANR-11-LABX-0056-LMH, LabEx LMH}

\author[E.\,W. Howe]{Everett W. Howe}
\address{Center for Communications Research,
         4320 Westerra Court,
         San Diego, CA 92129-1967 U.S.A.}
\email{\href{mailto:however@alumni.caltech.edu}{however@alumni.caltech.edu}}
\urladdr{\href{http://alumnus.caltech.edu/~however}{http://alumnus.caltech.edu/\lowtilde{}however}}

\author[C. Ritzenthaler]{Christophe Ritzenthaler}
\address{IRMAR, 
         Universit\'e de Rennes 1,
         Campus de Beaulieu,
         263 avenue du G\'en\'eral Leclerc,
         35042 Rennes Cedex, France.}
\email{\href{mailto:christophe.ritzenthaler@univ-rennes1.fr}{christophe.ritzenthaler@univ-rennes1.fr}}
\urladdr{\url{https://perso.univ-rennes1.fr/christophe.ritzenthaler/}}


\begin{abstract}
We give equations for $13$ genus-$2$ curves over~$\QQbar$, with models over~$\QQ$, whose unpolarized Jacobians are isomorphic to the square of an elliptic curve with complex multiplication by a maximal order. If the Generalized Riemann Hypothesis is true, there are no further examples of such curves. More generally, we prove under the Generalized Riemann Hypothesis that there exist exactly $46$ genus-$2$ curves over $\QQbar$ with field of moduli $\QQ$ whose Jacobians are isomorphic to the square of an elliptic curve with complex multiplication by a maximal order.
\end{abstract}

\maketitle

\section{Introduction} \label{sec:intro}

For $g>1$, let $\MM_g$ (resp.~$\Aa_g$) be the moduli space classifying absolutely irreducible projective smooth curves of genus~$g$ (resp.~principally polarized abelian varieties of dimension~$g$) over $\QQbar$. These spaces are quasi-projective varieties defined over~$\QQ$, linked by the Torelli map, which associates to a curve its Jacobian. To explain the modular interpretation of rational points on these spaces, we must define the terms \emph{field of definition} and \emph{field of moduli}. If $X$ is a curve or polarized abelian variety over $\QQbar$, we say that a field $F \subseteq \QQbar$ is a \emph{field of definition} of $X$ if there exists a variety~$X_0/F$ --- called a \emph{model} of $X$ over $F$ --- such that $X_0 \simeq_{\QQbar} X$. Since $\QQbar$ is a field of characteristic~$0$, by~\cite[Corollary~3.2.2, p.~54]{Koi72} we can define the \emph{field of moduli} of $X$ to be either
\begin{itemize}
\item the field fixed by the subgroup $\{\sigma \in \Gal(\QQbar/\QQ) ~|~ X \simeq X^{\sigma}\}$, or
\item the intersection of the fields of definition of $X$.
\end{itemize}
With these terms defined, we can say that the rational points on $\MM_g$ (resp.~$\Aa_g$) correspond to the isomorphism classes of curves (resp.~principally polarized abelian varieties) over $\QQbar$ that have field of moduli~$\QQ$~\cite{Bai62}.

There are a number of interesting sets of rational points on $\Aa_g$, but the complex multiplication (CM) abelian varieties --- that is, the principally polarized abelian varieties having endomorphism rings containing an order in a number field of degree~$2g$ over~$\QQ$ --- have attracted the most interest. When such a point on~$\Aa_g$ lies in the image of $\MM_g$, the corresponding curve is called a \emph{CM-curve}. For $g=2$, the set of \emph{simple} CM-abelian varieties with field of moduli $\QQ$ is known, and for those varieties that are Jacobians explicit equations have been computed for the corresponding curves~\cite{Spa94,Wam99,MU01,KS15,BS17}; for $g=3$ the similar set of possible CM maximal orders is determined in~\cite{Kil16} and conjectural equations for the curves are given in~\cite{Wen01,KW05,BILV16,LS16,KLL+17}. (And while we have avoided the case $g=1$ in the discussion above for technical reasons, it is still of course true that the CM-elliptic curves with rational $j$-invariants are known as well~\cite[Appendix~A.3]{Sil94}.)

In this article we consider genus-$2$ curves whose Jacobians are non-simple CM-abelian surfaces. Every such surface is isogenous to the square of a CM-elliptic curve, but we restrict our attention in two ways: first, we look only at surfaces that are \emph{isomorphic} (and not just isogenous) to $E^2$ for a CM-elliptic curve~$E$, and second, we only consider $E$ that have CM by a maximal order. The second restriction is not essential to our methods, and we impose it here in order to simplify some of our calculations. Note that if the elliptic curve $E$ has no CM --- \ie $\End(E)\simeq\ZZ$, then $E^2$ cannot be isomorphic to the Jacobian of a genus-$2$ curve, because $E^2$ has no indecomposable principal polarizations~\cite[Corollary~4.2, p.~159]{Lan05}.

\subsection*{Main Contributions.} 
We prove under the Generalized Riemann Hypothesis that there exist exactly $46$ genus-$2$ curves over $\QQbar$ with field of moduli $\QQ$ whose Jacobians are isomorphic to the square of an elliptic curve with CM by a maximal order. We show that among these $46$ curves exactly $13$ can be defined over $\QQ$, and we give explicit equations for them. In order to accomplish this, we develop an algorithm to compute, for an imaginary quadratic maximal order $\Ocal$, canonical forms for all positive definite unimodular Hermitian forms on $\Ocal\times\Ocal$. Such Hermitian forms correspond to principal polarizations $\varphi$ on $E^2$, and our algorithm computes the automorphism group of the polarized variety $(E^2,\varphi)$ and identifies the polarizations that come from genus-$2$ curves.

\subsection*{Related work.}
Hayashida and Nishi~\cite{HN65} consider in particular when a product of two elliptic curves, with CM by the same maximal order~$\Ocal$, is the Jacobian of a curve over~$\CC$, and they find that this happens if and only if the discriminant of~$\Ocal$ is different from $-1$, $-3$, $-7$, and $-15$. Hayashida~\cite{Hay68} gives the number of indecomposable principal polarizations on $E^2$ where $E/\CC$ is an elliptic curve with CM by a maximal order. More recently, Kani~\cite{Kan14,Kan16} gives existence results on Jacobians isomorphic to the product of two elliptic curves with control on the polarization, and Schuster~\cite{Sch90} and Lange~\cite{Lan05} study generalizations to higher dimensions. Rodriguez-Villegas~\cite{Rod00} considers the same situation as Hayashida and Nishi, and in the case where $\Ocal$ has class number~$1$ and odd discriminant, he gives an algorithm (relying on quaternion algebras) for producing curves with field of moduli~$\QQ$. Note finally that Fit\'e and Guitart~\cite{FG18} determine when there exists an abelian surface $A/\QQ$ that is $\QQbar$-isogenous to $E^2$, with $E/\QQbar$ a~CM-curve.

\subsection*{Outline.}
Our article proceeds as follows. Torelli's theorem (see \cite[Appendix]{LS01}) implies that our genus-$2$ curve $C$ has field of moduli $\QQ$ if and only if its principally polarized Jacobian $(E^2,\varphi)$ has field of moduli~$\QQ$. We therefore need to find all elliptic curves $E$ with CM by a maximal order $\Ocal$ and all polarizations $\varphi$ of $E^2$ such that $(E^2,\varphi)$ is isomorphic to all of its $\Gal(\QQbar/\QQ)$-conjugates. Proposition~\ref{prop:exp2} shows that if $E^2$ is isomorphic to all of its Galois conjugates --- even just as an abelian variety without polarization --- then the class group of $\Ocal$ has exponent at most~$2$. Under the Generalized Riemann Hypothesis, this gives us an explicit finite list of possible orders (Table~\ref{tab:order}). For each of these orders~$\Ocal$, one can identify the indecomposable principal polarizations $\varphi$ on $E^2$ and describe them as certain~\mbox{$2$-by-$2$} matrices $M$ with coefficients in $\Ocal$ (Proposition~\ref{prop:matM}). Tables of such matrices were computed by Hoffmann~\cite{Hof91} and Schiemann~\cite{Sch98} and were published online,\footnote{Available at \url{https://www.math.uni-sb.de/ag/schulze/Hermitian-lattices/}.} but they only include a fraction of the discriminants that we must consider. We therefore describe an algorithm, using a method different from that of Hoffmann and Schiemann, that we use to recompute these tables of matrices (Section~\ref{sec:HowPol}). Given such a matrix~$M$, we find explicit algebraic conditions on~$M$ for the principally polarized abelian surface~$(E^2,\varphi)$ to have field of moduli~$\QQ$ (Section~\ref{sec:cond}). We check whether these conditions are satisfied for each $M$ on our list.

We conclude the article with three more results: we heuristically compute the Cardona--Quer invariants~\cite{CQ05} of the associated curves~$C$ and see that the factorization of their denominators reveals interesting patterns; we show that the field of moduli is a field of definition if and only if $C$ has a non-trivial group of automorphisms (\ie of order greater than~$2$, see Section~\ref{prop:fielddef}); and for the curves $C$ defined over~$\QQ$, we compute equations and prove that they are correct.

\subsection*{Notation.}
In the following, $E$ is an elliptic curve over~$\QQbar$ with complex multiplication by a maximal order $\Ocal$ of discriminant $\Delta$ and with fraction field~$K$, which we sometimes call the \emph{CM-field}.

\section{Condition on \texorpdfstring{$E^2$}{E2}} \label{sec:woPol}

We are interested in the field of moduli $\Mb$ of a principally polarized abelian surface $(E^2,\varphi)$. As outlined above, we first consider the abelian surface $E^2$ alone and we give a necessary condition for $\Mb$ to be contained in the CM-field $K$. If~$\Mb\subseteq K$ then in particular we have $E^2\simeq(E^\sigma)^2$ for all $\sigma \in \Gal\left(\QQbar/K\right)$. The class group~$\Cl(\Ocal)$ acts simply transitively on the set of elliptic curves with CM by~$\Ocal$~\cite[Proposition~1.2, p.~99]{Sil94}. Since $\End(E^{\sigma})=\End(E)=\Ocal$, for each \mbox{$\sigma \in \Gal(\QQbar/K)$}, there exists a unique class of ideals $I_\sigma \in \Cl (\Ocal)$ such that~\mbox{$E^\sigma \simeq E / I_\sigma$}.

Using a result of Kani~\cite[Proposition~65, p.~335]{Kan11}, we get that, for $E$, $\sigma$ and~$I_\sigma$ defined as above, \[E^2 \simeq (E / I_\sigma)^2 \quad\Longleftrightarrow\quad I_\sigma^2 = \left[\Ocal\right],\] where the last equality is in $\Cl(\Ocal)$. Note that since we only work with maximal orders, the conditions on the conductors in Kani's result are trivially satisfied. Moreover by~\cite[Theorem~4.3, p.~122]{Sil94}, since for any $I \in \Cl(\Ocal)$ there exists $\sigma \in \Gal(\QQbar/K)$ (actually even in $\Gal\left(K\big(j(E)\big)/K\right)$) such that $E/I=E^{\sigma}$, we get the following proposition.

\begin{proposition} \label{prop:exp2}
A necessary condition for $\Mb \subseteq K$ is that the class group of $\Ocal$ has exponent at most $2$.
\end{proposition}

Louboutin~\cite{Lou90} shows that under the assumption of the Generalized Riemann Hypothesis, the discriminant $\Delta$ of an imaginary quadratic field whose class group is of exponent at most $2$ satisfies $|\Delta|\le 2\cdot 10^7$. In Table~\ref{tab:order} we list the $65$ fundamental discriminants satisfying this bound that give class groups of exponent at most~$2$.

{\small
\begin{table}[ht]
\centering
\begin{tabular}{cl}
\toprule
$\# \Cl(\Ocal)$ & \multicolumn{1}{c}{Discriminants $\Delta$} \\
\midrule
  $2^0$ &    $-3$,    $-4$,    $-7$,    $-8$,   $-11$,   $-19$,   $-43$,   $-67$, $-163$          \\[1ex]
  $2^1$ &   $-15$,   $-20$,   $-24$,   $-35$,   $-40$,   $-51$,   $-52$,   $-88$,  $-91$, $-115$, \\
        &  $-123$,  $-148$,  $-187$,  $-232$,  $-235$,  $-267$,  $-403$,  $-427$  \\[1ex]
  $2^2$ &   $-84$,  $-120$,  $-132$,  $-168$,  $-195$,  $-228$,  $-280$,  $-312$, \\
        &  $-340$,  $-372$,  $-408$,  $-435$,  $-483$,  $-520$,  $-532$,  $-555$, \\
        &  $-595$,  $-627$,  $-708$,  $-715$,  $-760$,  $-795$, $-1012$, $-1435$  \\[1ex]
  $2^3$ &  $-420$,  $-660$,  $-840$, $-1092$, $-1155$, $-1320$, $-1380$,          \\ 
        & $-1428$, $-1540$, $-1848$, $-1995$, $-3003$, $-3315$                    \\[1ex]
  $2^4$ & $-5460$ \\
\bottomrule
\end{tabular}
\bigskip
\caption{Discriminants $\Delta$ of the imaginary quadratic maximal orders $\Ocal$ of exponent at most~$2$, conditional on the Generalized Riemann Hypothesis.} \vspace{-5.5ex}
\label{tab:order}
\end{table}}

\section{Polarized abelian surfaces} \label{sec:Pol}

\subsection{Polarizations on the square of an elliptic curve}

We now consider the principal polarizations on the product surface $A = E^2$. A principal polarization on $A$ is, in particular, an isogeny of degree $1$ from $A$ to the dual $\Adual$ of $A$, but not every isomorphism $A\to\Adual$ is a principal polarization; other properties must be satisfied as well (see~\cite[\S~4.1]{BL04}). One such polarization is the product polarization $\varphi_0 = \varphi_E \times \varphi_E$. Given any other principal polarization $\varphi$, we can consider the automorphism $M = \varphi_0^{-1}\varphi$ of $A$, which (in light of the isomorphism $A = E^2$) we view as a matrix\footnote{All matrices in this paper act on the left.} in $\GL_2(\Ocal)$. Our first result characterizes the matrices that arise in this way; the statement is not new, but we provide a proof here because it introduces some of the ideas used in the sequel. (Recall~\cite[Exercise~7, p.~134]{Hal74} that two matrices $M_1$ and~$M_2$ in~$\GL_2(\Ocal)$ are said to be \emph{congruent} if there exists a matrix $P \in \GL_2(\Ocal)$ such that~$P^* M_1 P = M_2$, where $P^*$ is the conjugate transpose of~$P$.)

\begin{proposition} \label{prop:matM}
The map $M \mapsto \varphi_0 \cdot M$ defines a bijection between the positive definite unimodular Hermitian matrices with coefficients in $\Ocal$ and the principal polarizations on~$A$. Two principal polarizations are isomorphic to one another if and only if their associated matrices are congruent to one another.
\end{proposition}

\begin{proof}
By~\cite[Theorem~5.2.4, p.~121]{BL04}, the matrices $M$ corresponding to principal polarizations are totally positive symmetric endomorphisms of norm~$1$. Here the symmetry is with respect to the Rosati involution of $\End(A)$ associated to the polarization $\varphi_0$, which is the conjugate-transpose involution under the identification $\End(A) = M_2(\Ocal)$. Thus, the matrices $M$ corresponding to principal polarizations are exactly the positive definite unimodular Hermitian matrices.

Let $\varphi_1$ and $\varphi_2$ be two principal polarizations on $A$, corresponding to matrices~$M_1$ and $M_2$. The polarizations $\varphi_1$ and $\varphi_2$ are isomorphic to one another if and only if there exists an automorphism $\alpha\colon A\to A$ such that $\alphadual\varphi_1\alpha = \varphi_2$, where $\alphadual\colon\Adual\to\Adual$ is the dual of $\alpha$. This last condition is equivalent to
$(\varphi_0^{-1}\alphadual\varphi_0) (\varphi_0^{-1}\varphi_1\alpha) = \varphi_0^{-1}\varphi_2$. Now, $\varphi_0^{-1} \alphadual\varphi_0$ is nothing other than the Rosati involute of $\alpha$, so if we write $\alpha$ as a matrix $P \in \GL_2(\Ocal)$, the condition that determines whether $\varphi_1$ and $\varphi_2$ are isomorphic is simply $P^* M_1 P = M_2$.
\end{proof}

The principal polarizations on $A$ come in two essentially different types.

\begin{definition}
A polarization $\varphi$ on an abelian variety $A$ over a field $k$ is said to be \emph{geometrically decomposable} if there exist two abelian varieties $A_1$ and $A_2$ over $\bar{k}$ of positive dimension, together with polarizations $\varphi_1$ and $\varphi_2$, such that $(A,\varphi)$ and $(A_1 \times A_2, \varphi_1 \times \varphi_2)$ are isomorphic over $\bar{k}$. A polarization that is not geometrically decomposable is \emph{geometrically indecomposable}. For brevity's sake, in this paper we drop the adjective \emph{geometrically} and simply use the terms \emph{decomposable} and \emph{indecomposable} for these concepts.
\end{definition}

Results in~\cite{Wei57,Hoy63,OU73} show that a principally polarized abelian surface is the Jacobian of a curve if and only if the polarization is indecomposable. In the remainder of this section we show how we can easily compute representatives for the congruence classes of matrices representing the decomposable polarizations on~$E^2$; we focus on the indecomposable polarizations in later sections.

\begin{proposition} \label{prop:decomposable}
If $\varphi$ is a decomposable polarization on $E^2$, then there exist elliptic curves $F$ and $F'$ that have CM by $\Ocal$ such that $\varphi$ is the pullback to $E^2$ of the product polarization on $F\times F'$ via some isomorphism $E^2 \simeq F\times F'$. The pair~$(F,F')$ giving rise to a given decomposable polarization is unique up to interchanging $F$ and~$F'$ and up to isomorphism for each elliptic curve. Moreover, for every~$F$ with CM by $\Ocal$ there exists an $F'$ with CM by $\Ocal$ such that $E^2 \simeq F\times F'$.
\end{proposition}

\begin{proof}
First we note that by definition, if $\varphi$ is a decomposable polarization on $E^2$ there must exist elliptic curves $F$ and $F'$, isogenous to $E$, such that $\varphi$ is the pullback of the product polarization on $F\times F'$ under some isomorphism $E^2\simeq F\times F'$. Now, the center of $\End(E^2)$ is $\End(E) = \Ocal$, while the center of $\End(F\times F')$ is~$\End(F)\cap\End(F')$; since $\Ocal$ is a maximal order, $F$ and $F'$ both have CM by~$\Ocal$.

If $(\alpha,\beta)\colon G\to F\times F'$ is an embedding of an elliptic curve $G$ into $F\times F'$, then the pullback of the product polarization to $G$ is the morphism
\[\begin{bmatrix} \hat\alpha & \hat\beta\end{bmatrix} \begin{bmatrix} 1& 0 \\ 0 & 1\end{bmatrix} \begin{bmatrix} \alpha \\ \beta\end{bmatrix}
 = \hat\alpha \alpha + \hat\beta \beta = \deg(\alpha) + \deg(\beta);\]
that is, the pullback is the multiplication-by-$d$ map, with ${d = \deg(\alpha) + \deg(\beta)}$. It follows that if $\varphi$ is the pullback to $E^2$ of the product polarization on $F\times F'$ via some isomorphism $E^2 \simeq F\times F'$, then the set of elliptic curves~$G$ for which there exists an embedding $\epsilon\colon G\to E^2$ such that $\epsilon^*\varphi$ is a principal polarization is simply~$\{F,F'\}$. Thus, for a given decomposable principal polarization, the pair $(F,F')$ is unique up to order and isomorphism. 

As we noted at the beginning of Section~\ref{sec:woPol}, the set of elliptic curves with CM by~$\Ocal$ is a principal homogenous space for the class group of $\Ocal$. Given an $F$ with CM by $\Ocal$, let~$I\in\Cl(\Ocal)$ be the ideal class that takes $E$ to $F$. If $F'$ is an elliptic curve with CM by $\Ocal$, say corresponding to an ideal class $I'\in\Cl(\Ocal)$, then $E^2 \simeq F\times F'$ if and only if $I'$ is the inverse of $I$ (see~\cite[Proposition~65, p.~335]{Kan11}). This proves the final statement of the proposition.
\end{proof}

\begin{corollary} \label{cor:decomposable}
Let $h$ denote the class number of $\Ocal$, and let $t$ denote the size of the~$2$-torsion subgroup of the class group. The number of decomposable polarizations on $E^2$ is equal to $(h+t)/2$. 
\end{corollary}

\begin{proof}
The proof of Proposition~\ref{prop:decomposable} shows that the unordered pairs $(F,F')$ with $E^2 \simeq F\times F'$ correspond to unordered pairs $(I,I^{-1})$, where $I\in \Cl(\Ocal)$. The number of such pairs is $(h+t)/2$.
\end{proof}

Let $F$ be an elliptic curve with CM by $\Ocal$ and let $I$ be the ideal class that takes~$E$ to $F$. Let $\afrak$ be an ideal of $\Ocal$ representing~$I$, such that $\afrak$ is not divisible by any nontrivial ideal of $\ZZ$. We may write $\afrak = (n,\alpha)$, where $n=\Norm(\afrak)\in\ZZ$ and where~$\alpha\in\afrak$ is chosen so that the ideal $\alpha\afrak^{-1}$ is coprime to $n\Ocal$; then there exist~$x,y\in\ZZ$ such that $xn^2 - y\Norm(\alpha) = n$. Let~$F'$ be the elliptic curve such that $E^2 \simeq F\times F'$. We prove the following corollary in~Section~\ref{sec:cond}.

\begin{corollary} \label{cor:explicit-decomposable}
In the notation of the paragraph above, the isomorphism class of the decomposable polarization on $E^2$ obtained from pulling back the product polarization on $F\times F'$ is represented by the congruence class of the matrix \[\begin{pmatrix} n + \frac{\Norm(\alpha)}{n} & (x + y)\alpha \\ (x + y)\bar{\alpha} & x^2 n + y^2\frac{\Norm(\alpha)}{n}\end{pmatrix}.\]
\end{corollary}

\subsection{How to find the polarizations?} \label{sec:HowPol}

In Section~\ref{sec:woPol}, we identified 65 orders~$\Ocal$ for which we need to compute the set of indecomposable principal polarizations, or equivalently, representatives of the congruence classes of indecomposable positive definite unimodular Hermitian matrices with coefficients in $\Ocal$. In this section we describe how we computed these representatives.

Fix an embedding $\epsilon_0$ of $K$ into the complex numbers. For any $\alpha\in\Ocal$, we write~$\alpha>0$ if either the trace of $\alpha$ is positive, or the trace of $\alpha$ is $0$ and $\epsilon_0(\alpha)$ has positive imaginary part. Then for $\alpha,\beta\in\Ocal$ we write $\alpha>\beta$ if $\alpha-\beta > 0$. Clearly this gives us a total ordering on $\Ocal$.

Let $\Hcal$ denote the set of positive definite unimodular Hermitian matrices with coefficients in $\Ocal$. Let $\chi\colon\Hcal\to\NN\times\NN\times\Ocal$ be the map that sends a matrix
$M = \left(\begin{smallmatrix} a&b\\ \bar{b}&d \end{smallmatrix}\right)$ to the triple $(a,d,b)$. We define a total ordering on $\Hcal$ by saying that $M_1 < M_2$ if~$\chi(M_1) < \chi(M_2)$ in the lexicographic ordering on $\NN\times\NN\times\Ocal$.

Given any $M\in\Hcal$, we say that $M$ is \emph{reduced} if $M\le M'$ for all $M'$ congruent to~$M$. Clearly every $M\in\Hcal$ is congruent to a unique reduced matrix. The following algorithm produces the reduced matrix that is congruent to a given~$M$.

\begin{algorithm} \label{alg:reduced-form}
\begin{algtop}
\algin A positive definite unimodular Hermitian matrix $M$ with coefficients
       in $\Ocal$, specified by $a,d\in\ZZ$ and $b\in\Ocal$ such that
       $M = \left(\begin{smallmatrix}a&b\\\bar{b}&d\end{smallmatrix}\right)$.
\algout The reduced matrix congruent to $M$.
\end{algtop}
\begin{alglist}
\item Set $a' = 1$.
\item \label{reduced-short1}
      Compute the set $A'$ of vectors
      $\xx = (x_1,x_2) \in\Ocal^2$ such that $\xx^* M \xx = a'$
      and such that $x_1$ and $x_2$ generate the unit ideal of $\Ocal$. 
      If $A'=\emptyset$, increment $a'$ and repeat.
\item Set $d'= a'$.
\item \label{reduced-short2}
      Compute the set $D'$ of vectors
      $\yy = (y_1,y_2) \in\Ocal^2$ such that $\yy^* M \yy = d'$
      and such that $y_1$ and $y_2$ generate the unit ideal of $\Ocal$.
      If $D'=\emptyset$, increment $d'$ and repeat.
\item Initialize $\Mcal$ to be the empty set.
\item For each $\xx \in A'$ and $\yy \in D'$ such that $\xx$ and $\yy$ 
      generate $\Ocal^2$ as an $\Ocal$-module, 
      let~$M'$ be the matrix representing the Hermitian form $M$ 
      written on the basis $\xx$, $\yy$ of $\Ocal^2$, 
      and add $M'$ to the set $\Mcal$.
\item If $\Mcal$ is empty, increment $d'$ and return to 
      Step~\eqref{reduced-short2}. 
\item Find the smallest element $M'$ of $\Mcal$ under the ordering
      of $\Hcal$ defined above.
\item Output $M'$.
\end{alglist}
\end{algorithm}

\begin{remark} \label{rem:vect}
In Steps~\eqref{reduced-short1} and~\eqref{reduced-short2} of Algorithm~\ref{alg:reduced-form}, we need to find vectors in $\Ocal^2$ of a given length under the quadratic form specified by $M$. We note that this is a finite computation: if $\xx = (x_1, x_2)$ satisfies $\xx^* M \xx = n$, with $M = \left(\begin{smallmatrix}a&b\\\bar{b}&d\end{smallmatrix}\right)$, then \[\Norm(ax_1 + bx_2) + \Norm(x_2) = an.\]
Thus, to solve $\xx^* M \xx = n$, we can simply enumerate all pairs $(u,v)\in\Ocal^2$ with $\Norm(u)+\Norm(v) = an$, and keep those pairs for which $u-bv$ is divisible by~$a$.

Note that solving $\xx^* M \xx = n$ can be done more quickly when the value of $a$ is small. Thus, in Algorithm~\ref{alg:reduced-form}, once one finds a short vector $\xx = (x_1,x_2)$ with~$x_1$ and $x_2$ coprime, it is worthwhile to compute \emph{any} vector $\yy$ such that $\xx$ and $\yy$ generate~$\Ocal$, and to replace $M$ with the congruent form obtained by rewriting $M$ on the basis $\xx,\yy$.
\end{remark}

\begin{theorem} \label{thm:reduced-form}
Algorithm~\textup{\ref{alg:reduced-form}} terminates with the correct result.
\end{theorem}

\begin{proof}
Let $M' = \big(\begin{smallmatrix}a'&b'\\\bar{b}'&d'\end{smallmatrix}\big)$ be the reduced matrix congruent to $M$. If $P = \left(\begin{smallmatrix}x_1&y_1\\x_2&y_2\end{smallmatrix}\right)$ is an element of $\GL_2(\Ocal)$ such that $P^* M P = M'$, and if we set $\xx = (x_1,x_2)$ and~$\yy = (y_1,y_2)$, then $a' = \xx^* M \xx$ and $d' = \yy^* M \yy$. By the very definition of the ordering on $\Hcal$, then, we want to find vectors $\xx$ and $\yy$, each with coordinates that are coprime to one another, such that $\xx^*M\xx$ is as small as possible and $\yy^*M\yy$ is as small as possible, given that $\xx$ and $\yy$ generate $\Ocal^2$ as an $\Ocal$-module. This is what the algorithm does. Finally, among all possible such pairs $(\xx,\yy)$, we simply need to choose the one that gives the smallest matrix.
\end{proof}

Hayashida~\cite{Hay68} gives a formula for the number of isomorphism classes of indecomposable principal polarizations on $E^2$ in the case where $E$ has CM by a maximal order.\footnote{There is a typographical error in Hayashida's paper. In the second line of page 43, the term $(1/4)(1 - (-1))^{(m^2-1)/8}$ should be $(1/4)(1 - (-1)^{(m^2-1)/8})h$. Note that the correction involves both moving a parenthesis and adding an instance of the variable~$h$.} Hayashida's proof does not immediately lead to a constructive method of finding polarizations representing the isomorphism classes, but simply knowing the number of isomorphism classes is the key to a straightforward algorithm for producing such representatives.

\begin{algorithm} \label{alg:all-polarizations}
\begin{algtop}
\algin A fundamental discriminant $\Delta<0$.
\algout A list of reduced matrices representing the distinct congruence classes of positive definite unimodular Hermitian matrices with entries in the order $\Ocal$ of discriminant $\Delta$, separated into the decomposable and indecomposable classes.
\end{algtop}
\begin{alglist}
\item Compute the number $N$ of indecomposable polarizations on $E^2$ using
      Hayashida's formula.
\item Compute the set $\Dcal$ of reduced matrices representing decomposable
      polarizations, using Corollary~\textup{\ref{cor:explicit-decomposable}}
      and Algorithm~\textup{\ref{alg:reduced-form}}.
\item Initialize $\Ical$ to be the empty set and set $P = 0$.
\item \label{all-product}
      Increment $P$, and 
      compute the set $S$ of elements of $\Ocal$ of norm $P-1$.
\item For every divisor $a$ of $P$ with $a\le P/a$, and for every $b \in S$\textup{:}
      \begin{algsublist}
      \item Compute the reduced form $M$ of the matrix 
            $\left(\begin{smallmatrix}a&b\\\bar{b}&P/a\end{smallmatrix}\right)$.
      \item If $M$ is not contained in $\Dcal\cup \Ical$, 
            then add $M$ to the set $\Ical$.
      \end{algsublist}
\item If $\#\Ical < N$, then return to Step~\eqref{all-product}.
\item Return $\Dcal$ and $\Ical$.
\end{alglist}
\end{algorithm}

Of course, for our goal of producing genus-$2$ curves over $\QQ$ with Jacobians isomorphic to $E^2$, we only need the indecomposable polarizations.

\begin{theorem} \label{thm:all-polarizations}
Algorithm~\textup{\ref{alg:all-polarizations}} terminates with the correct result.
\end{theorem}

\begin{proof}
The algorithm is very straightforward. Every isomorphism class of principal polarization appears somewhere on the countable list that we are considering, and we simply enumerate the polarizations and compute their reduced forms until we have found the right number of isomorphism classes.
\end{proof}

\begin{remark}
In our applications, when the class group of $\Ocal$ has exponent at most~$2$, we can speed up our algorithm as follows: once we have a principal polarization $M$ on $E^2$, we can view the same matrix as giving a polarization on $F^2$ for any elliptic curve $F$ with CM by $\Ocal$. Since the class group has exponent at most~$2$, there exists an isomorphism $E^2 \to F^2$, and pulling $M$ back to $E^2$ via such an isomorphism gives a new positive definite unimodular Hermitian matrix~$M'$. Each time we find a new reduced polarization~$M$, we compute the reduced forms of the polarizations~$M'$ associated to all the curves $F$ isogenous to~$E$, and add these reduced forms to the set $\Dcal$ if they are new.
\end{remark}

If $\varphi$ is a principal polarization on $E^2$ and $M$ is the corresponding Hermitian matrix, then the automorphism group of the polarized abelian variety $(E^2,\varphi)$, denoted by $\Aut(E^2,\varphi)$, is isomorphic to the group $\{ P \in \GL_2(\Ocal) ~|~ P^* M P=M\}$. Note that if $\varphi$ is indecomposable, so that $(E^2,\varphi)$ is the polarized Jacobian of a curve~$C$, then Torelli's theorem~\cite[Appendix]{LS01} shows that this group is also isomorphic to $\Aut(C)$. In any case, computing $\Aut(E^2,\varphi)$ is straightforward:

\begin{algorithm} \label{alg:automorphisms}
\begin{algtop}
\algin A positive definite unimodular Hermitian matrix 
       $M = \left(\begin{smallmatrix}a&b\\\bar{b}&d\end{smallmatrix}\right)$
       with entries in an imaginary quadratic maximal order $\Ocal$.
\algout A list of all matrices $P \in \GL_2(\Ocal)$ such that $P^* M P=M$.
\end{algtop}
\begin{alglist}
\item Compute the set $A$ of vectors 
      $\xx = (x_1,x_2) \in\Ocal^2$ such that $\xx^* M \xx = a$
      and such that $x_1$ and $x_2$ generate the unit ideal of $\Ocal$.
\item Compute the set $D$ of vectors 
      $\yy = (y_1,y_2) \in\Ocal^2$ such that $\yy^* M \yy = d$
      and such that $y_1$ and $y_2$ generate the unit ideal of $\Ocal$.
\item Initialize $\Acal$ to be the empty set.
\item For each $\xx \in A$ and $\yy \in D$ such that $\xx$ and $\yy$ 
      generate $\Ocal^2$ as an $\Ocal$-module\textup{:}
      \begin{algsublist}
      \item Compute $b' = \xx^* M \yy$.
      \item If $b' = b$ then add the matrix
      $\left(\begin{smallmatrix}x_1&y_1\\x_2&y_2\end{smallmatrix}\right)$
      to the set $\Acal$.
      \end{algsublist}
\item Output $\Acal$.
\end{alglist}
\end{algorithm}

\noindent(See Remark~\ref{rem:vect} for an explanation of how to implement the two first steps.)

\begin{theorem} \label{thm:automorphisms}
Algorithm~\textup{\ref{alg:automorphisms}} terminates with the correct result.
\end{theorem}

\begin{proof}
If $P=\left(\begin{smallmatrix} x_1&y_1\\ x_2&y_2\end{smallmatrix}\right)\in\GL_2(\Ocal)$ satisfies $P^* M P=M$, then $\xx = (x_1,x_2)$ and $\yy = (y_1,y_2)$ are vectors in $\Ocal^2$ such that $\xx^* M \xx = a$ and $\yy^* M \yy = d$ and $\xx^* M \yy = b$. The algorithm simply enumerates all $\xx$ and $\yy$ that meet the first two conditions, and checks to see whether they meet the third.
\end{proof}

\subsection{Conditions on the polarization} \label{sec:cond}

Throughout this section, $E$ is an elliptic curve with CM by a maximal order $\Ocal$ of an imaginary quadratic field $K$ whose class group has exponent at most~$2$. Also $\varphi$ is a principal polarization on $E^2$ corresponding (as in Proposition~\ref{prop:matM}) to a positive definite unimodular Hermitian matrix $M$ with entries in $\Ocal$ and $\Mb$ is the field of moduli of the polarized abelian variety $(E^2,\varphi)$. We resume our analysis of the condition that $\Mb = \QQ$.

\begin{proposition} \label{prop:polarization}
Let $\afrak_1$, $\dotsc$, $\afrak_h$ be ideals of $\Ocal$ representing all of the elements of the class group of $\Ocal$, and for each $i$ let $n_i\in \ZZ_{>0}$ generate $\Norm(\afrak_i)$. Then $\Mb = \QQ$ if and only if for every $i$ there exists a matrix $P_i \in \GL_2(K)$, with entries in $\afrak_i$, such that $n_i M = P_i^* M P_i$.
\end{proposition}

\begin{proof}
Lemma~\ref{lemma:KvsQ} below shows that $\Mb = \QQ$ if and only if $\Mb\subseteq K$, and this is the case if and only if for every $\sigma\in\Gal(\QQbar/K)$ there exists an isomorphism $\alpha_\sigma\colon (E^2,\varphi) \to \big((E^\sigma)^2,\varphi^{\sigma}\big)$. To understand this condition, we use the classical theory of complex multiplication of abelian varieties; the book of Shimura and Taniyama~\cite{Shi61} is one possible reference, especially Chapter~II.

Under the embedding $\epsilon_0\colon K\to \CC$ we chose earlier, the isomorphism classes of elliptic curves over $\QQbar\subset\CC$ with CM by $\Ocal$ correspond to the lattices $\epsilon_0(\afrak)$ up to scaling, for fractional ideals $\afrak$ of $\Ocal$. Since the class group of the order $\Ocal$ is $2$-torsion, we have $E^2 \simeq F^2$ for every $E$ and $F$ with CM by $\Ocal$, so we may as well choose our~$E$ so that it corresponds to the trivial ideal~$\Ocal$.

Let $\Delta$ be the discriminant of $\Ocal$ and let $\delta\in\Ocal$ be a square root of $\Delta$, chosen so that $\epsilon_0(\delta)$ is positive imaginary. Note that the trace dual $\afrak^\dagger$ of an arbitrary fractional $\Ocal$-ideal $\afrak$ is $(1/\delta)\afrak^{-1}$. If $F$ is the elliptic curve corresponding to $\afrak$, then the dual of $F$ is the elliptic curve corresponding to the complex conjugate of $\afrak^\dagger$, and the canonical principal polarization of $F$ is the isomorphism $\afrak\to (1/\delta)\bar{\afrak}^{-1}$ given by $x\mapsto x/(n\delta)$, where $n\in\QQ$ is the positive generator of $\Norm(\afrak)$. (See~\cite[\S~6.3]{Shi61} for more details.) 

Let $\varphi_0$ be the product polarization on $E^2$. For $\alpha_\sigma\colon E^2\to (E^\sigma)^2$ to give an isomorphism between $(E^2,\varphi)$ and $\big((E^\sigma)^2,\varphi^\sigma\big)$, the following diagram must be commutative:
\[\xymatrix{ E^2 \ar[rr]^{M} \ar[d]_{\alpha_\sigma} && E^2\ar[rr]^{\varphi_0} && \widehat{E}^2\\ (E^\sigma)^2 \ar[rr]^{M} && (E^\sigma)^2 \ar[rr]^{\varphi_0^\sigma} && (\widehat{E}^\sigma)^2\ar[u]_{ \widehat{\alpha}_{\sigma}}.}\]
To express this diagram in terms of lattices, we let $\afrak$ be an ideal corresponding to~$E^\sigma$, we let $n = \Norm(\afrak)$, and we let $P_\sigma$ be the matrix in $\GL_2(K)$ corresponding to~$\alpha_\sigma$. Then the preceding diagram becomes
\[\xymatrix{ \Ocal\times\Ocal \ar[rr]^{M} \ar[d]_{P_\sigma} && \Ocal\times\Ocal \ar[rr]^{1/\delta\qquad} && (1/\delta)(\Ocal\times\Ocal)\\
\afrak\times\afrak \ar[rr]^{M} && \afrak\times\afrak \ar[rr]^{1/(n\delta)\qquad} && (1/\delta)(\bar{\afrak}^{-1}\times\bar{\afrak}^{-1}) \ar[u]_{ P^*_{\sigma}}.}\]
Thus, there exists an isomorphism $(E^2,\varphi)\to \big((E^\sigma)^2,\varphi^\sigma\big)$ of polarized varieties if and only if there exists a matrix $P$, with entries in $\afrak$, such that $n M = P^* M P$. Since the Galois group of $\QQbar/K$ acts transitively on the set of elliptic curves with~CM by~$\Ocal$, the field of moduli of $(E^2,\varphi)$ is contained in $K$ if and only if we can find such a matrix $P$ for each of the ideals $\afrak_1$, $\dotsc$, $\afrak_h$.
\end{proof}

\begin{lemma} \label{lemma:KvsQ}
Let $E$, $\varphi$, and $\Mb$ be as mentioned at the beginning of this section. Then $\Mb = \QQ$ if and only if $\Mb\subseteq K$.
\end{lemma}

\begin{proof}
Let us assume that $\Mb\subseteq K$; we must show that $\Mb = \QQ$. Since $\Ocal$ has a class group of exponent at most~$2$, \cite[Exercise~5.8, p.~124]{Shi71} implies that $\QQ\big(j(E)\big)$ is totally real. Let $\iota$ be any complex conjugation in~$\Gal(\QQbar/\QQ)$, so that $\iota$ acts trivially on $\QQ(j(E))$ and nontrivially on~$K$. Given any $\sigma\in \Gal(\QQbar/\QQ)$, we want to show that $(E^2,\varphi)\simeq((E^\sigma)^2,\varphi^\sigma)$.

If $\sigma$ acts trivially on~$K$, then such an isomorphism exists, because $\Mb\subseteq K$. Otherwise, $\sigma\iota$ acts trivially on $K$, and we have $(E^2,\varphi)\simeq ((E^{\sigma\iota})^2,\varphi^{\sigma\iota})$, and therefore $((E^{\iota})^2,\varphi^{\iota})\simeq ((E^\sigma)^2,\varphi^\sigma)$. So it is enough for us to show that \mbox{$(E^2,\varphi)\simeq((E^{\iota})^2,\varphi^{\iota})$}. If we choose our model of $E$ to be defined over $\QQ(j(E))$, then $E^\iota = E$, and we simply need to show that there exists an element $P$ of $\GL_2(\Ocal)$ such that~$\bar{M} = P^*MP$. If $M = \left(\begin{smallmatrix} a& b \\\bar{b} & d\end{smallmatrix}\right)$, we can simply take $P= \big(\begin{smallmatrix} b & d \\ -a & -\bar{b} \end{smallmatrix}\big)$.
\end{proof}

At this point, we have reviewed enough CM theory to prove Corollary~\ref{cor:explicit-decomposable}.

\begin{proof}[Proof of Corollary~\textup{\ref{cor:explicit-decomposable}}]
We are given an ideal $\afrak = (n,\alpha)$ of $\Ocal$, where $n\in\ZZ$ is the norm of $\afrak$ and where $\alpha\in\Ocal$, and we have $x,y\in\ZZ$ such that $xn^2 - y\Norm(\alpha) = n$. The complex conjugate $\bar{\afrak}$ of $\afrak$ represents the inverse of the class of $\afrak$ in $\Cl(\Ocal)$, and the matrix $P = \left(\begin{smallmatrix}n & y\alpha \\ \bar{\alpha} & xn\\\end{smallmatrix}\right)$ takes the lattice $\Ocal\times\Ocal\subset K^2$ onto the lattice $\afrak\times\bar{\afrak}$. The dual lattice for $\afrak\times\bar{\afrak}$ is $(n\delta)^{-1} \cdot (\afrak\times\bar{\afrak})$ (where $\delta$ is the positive imaginary square root of~$\Delta$ as in the proof of Proposition~\ref{prop:polarization}) and the product polarization from~$\afrak\times\bar{\afrak}$~to~its~dual is simply multiplication by $1/(n\delta)$. Pulling this polarization back to $\Ocal\times\Ocal$ via~$P$~gives us the polarization $(n\delta)^{-1} P^* P$. Since the product polarization on $\Ocal\times\Ocal$ is $1/\delta$, the pullback polarization is represented by the endomorphism $(1/n)P^*P$ of $\Ocal\times\Ocal$, and we compute that $(1/n) P^* P$ is the matrix given in the statement of the corollary.
\end{proof}

We close this section by indicating how we can check the criterion given in Proposition~\ref{prop:polarization}: namely, given the polarization matrix $M$ and an ideal $\afrak$ with $\Norm(\afrak) = n\ZZ$, how can we determine whether there exists a matrix $P\in M_2(\afrak)$ that satisfies $n M = P^* M P$?

Suppose there exists such a matrix $P$. If $M=\left(\begin{smallmatrix} a & b \\ \bar{b} & d \end{smallmatrix}\right)$ let us take $L = \left(\begin{smallmatrix} a & b \\ 0 & 1 \end{smallmatrix}\right)$, so that $L^*L = aM$. Let $Q = LPL^{-1}$. Then the condition $n M = P^* M P$ becomes the condition $n\Id = Q^*Q$. This equality can only hold if $Q$ is of the form
\[Q = \begin{pmatrix} x & y \\ z & t \end{pmatrix} \in GL_2(K)\] 
where $x,y,z,t \in K$ satisfy $\Norm(x) + \Norm(z) = \Norm(y) + \Norm(t) = n$ and $\bar{x}y + \bar{z}t = 0$. Since we have 
\[P = L^{-1} Q L = \begin{pmatrix} x - bz & \frac{bx + y - b^2 z - bt}{a}\\ az & bz + t \end{pmatrix} \in M_2(\afrak),\]
we see that we must have $x=X/a$, $y=Y/a$, $z=Z/a$, and $t=T/a$ with $X,Y,Z,T \in \afrak$.

Therefore, to check whether a matrix $P$ with the desired properties exists, it suffices to compute and store all solutions $\left(X,Z\right) \in \afrak\times\afrak$ to the norm equation $\Norm(X) + \Norm(Z) = a^2n$ (which can be done efficiently). Then, for every two solutions $(X,Z)$ and $(Y,T)$ satisfying $\bar{X} Y + \bar{Z} T =0$, we can check whether the corresponding matrix $P$ lies in $M_2(\afrak)$. If we obtain such a $P$ for each of the ideals~$\afrak_i$ from Proposition~\ref{prop:polarization}, then the field of moduli for $(E^2,\varphi)$ is~$\QQ$. In fact, we need only find a $P$ for each $\afrak_i$ in a set that generates the class group of~$\Ocal$.

\subsection{Results}

We have implemented the algorithms described in the previous sections. We were able to test all polarizations on the $65$ possible orders identified in Section~\ref{sec:woPol}. The results are presented in Table~\ref{tab:res}.

{\small
\begin{table}[ht]
\centering
\begin{tabular}{r@{\hbox to 0.5em{}}rrcrr@{\hbox to 0.5em{}}rrcrrrrc}
\toprule
$h$ & $\Delta$ & $\#\varphi$ & $\#C$\hbox to 0.5em{} &&
$h$ & $\Delta$ & $\#\varphi$ & $\#C$\hbox to 0.5em{} &&
$h$ & $\Delta$ & $\#\varphi$ & $\#C$\hbox to 0.5em{} \\
\cmidrule(l{0.25em}r{0.75em}){1-4}\cmidrule(l{0.25em}r{0.75em}){6-9}\cmidrule(l{0.75em}r{0.75em}){11-14}
$1$ &    $-3$ &   $0$ & $0$ && $4$ &   $-84$ &   $2$ & $0$ && $8$  &  $-420$ &  $10$ & $0$ \\
    &    $-4$ &   $0$ & $0$ &&     &  $-120$ &   $5$ & $3$ &&      &  $-660$ &  $16$ & $0$ \\
    &    $-7$ &   $0$ & $0$ &&     &  $-132$ &   $3$ & $1$ &&      &  $-840$ &  $22$ & $0$ \\
    &    $-8$ &   $1$ & $1$ &&     &  $-168$ &   $4$ & $0$ &&      & $-1092$ &  $22$ & $0$ \\
    &   $-11$ &   $1$ & $1$ &&     &  $-195$ &   $8$ & $0$ &&      & $-1155$ &  $32$ & $0$ \\
    &   $-19$ &   $1$ & $1$ &&     &  $-228$ &   $5$ & $1$ &&      & $-1320$ &  $36$ & $0$ \\
    &   $-43$ &   $2$ & $2$ &&     &  $-280$ &  $14$ & $0$ &&      & $-1380$ &  $34$ & $0$ \\
    &   $-67$ &   $3$ & $3$ &&     &  $-312$ &  $11$ & $1$ &&      & $-1428$ &  $28$ & $0$ \\
    &  $-163$ &   $7$ & $7$ &&     &  $-340$ &  $14$ & $0$ &&      & $-1540$ &  $46$ & $0$ \\
    &         &       &     &&     &  $-372$ &   $8$ & $0$ &&      & $-1848$ &  $46$ & $0$ \\
$2$ &   $-15$ &   $0$ & $0$ &&     &  $-408$ &  $14$ & $0$ &&      & $-1995$ &  $56$ & $0$ \\
    &   $-20$ &   $1$ & $1$ &&     &  $-435$ &  $16$ & $0$ &&      & $-3003$ &  $72$ & $0$ \\
    &   $-24$ &   $1$ & $1$ &&     &  $-483$ &  $12$ & $0$ &&      & $-3315$ & $128$ & $0$ \\
    &   $-35$ &   $2$ & $0$ &&     &  $-520$ &  $25$ & $3$ &&      &         &       &     \\
    &   $-40$ &   $2$ & $2$ &&     &  $-532$ &  $14$ & $0$ && $16$ & $-5460$ & $128$ & $0$ \\
    &   $-51$ &   $2$ & $0$ &&     &  $-555$ &  $20$ & $0$ &&      &         &       &     \\
    &   $-52$ &   $2$ & $2$ &&     &  $-595$ &  $28$ & $2$ &&      &         &       &     \\
    &   $-88$ &   $4$ & $2$ &&     &  $-627$ &  $16$ & $0$ &&      &         &       &     \\
    &   $-91$ &   $4$ & $0$ &&     &  $-708$ &  $15$ & $1$ &&      &         &       &     \\
    &  $-115$ &   $6$ & $0$ &&     &  $-715$ &  $36$ & $0$ &&      &         &       &     \\
    &  $-123$ &   $4$ & $0$ &&     &  $-760$ &  $41$ & $1$ &&      &         &       &     \\
    &  $-148$ &   $5$ & $3$ &&     &  $-795$ &  $28$ & $2$ &&      &         &       &     \\
    &  $-187$ &   $8$ & $0$ &&     & $-1012$ &  $28$ & $0$ &&      &         &       &     \\
    &  $-232$ &   $9$ & $5$ &&     & $-1435$ &  $64$ & $0$ &&      &         &       &     \\
    &  $-235$ &  $12$ & $0$ &&     &         &       &     &&      &         &       &     \\
    &  $-267$ &   $8$ & $0$ &&     &         &       &     &&      &         &       &     \\
    &  $-403$ &  $18$ & $0$ &&     &         &       &     &&      &         &       &     \\
    &  $-427$ &  $16$ & $0$ &&     &         &       &     &&      &         &       &     \\
\bottomrule
\end{tabular}
\bigskip
\caption{The number of indecomposable principal polarizations~ $\varphi$ and the number of isomorphism classes of curves $C$ with field of moduli $\QQ$ for each discriminant $\Delta$, grouped by class number~$h$.}
\label{tab:res}
\end{table}}

There exist $1226$ indecomposable polarizations, in total. Our algorithms, implemented in \verb?Magma? on a laptop with a 2.50 GHz Intel Core i7-4710MQ processor, took less than $21$ minutes to compute all of the polarizations; about $10$ minutes of that time was spent on the largest discriminant. The computation required about~$2.8$~GB of memory.

Once we computed the polarizations, it took about $26$ minutes (on the same laptop) to check the conditions of Proposition~\ref{prop:polarization}. For this calculation, the largest discriminant represented more than two-thirds of the computation time.

In the end, we obtained exactly $46$ polarizations $\varphi$ such that the principally polarized abelian surface $(E^2,\varphi)$ is isomorphic to the Jacobian of a curve $C$ with field of moduli $\QQ$.
These $46$ curves are obtained only from orders whose class groups have order $1$, $2$, or $4$.

\section{Computation of invariants and final remarks} \label{sec:FinRes}

\subsection{Invariants of the genus-\texorpdfstring{$2$}{2} curves \texorpdfstring{$C$}{C}} 

A genus-$2$ curve $C$ has field of moduli~$\QQ$ if and only if all of its absolute invariants are defined over $\QQ$ (see for example~\cite[\S~3]{LRS12}). This is in particular true for the triplet $(g_1,g_2,g_3)$ of invariants defined by Cardona and Quer in~\cite{CQ05}, which characterizes a genus-$2$ curve up to~$\QQbar$-isomorphism and enables one to find an equation $y^2=f(x)$ for the curve. We quickly review here a strategy for obtaining the Cardona--Quer invariants for the~$46$ curves whose invariants are $\QQ$-rational.

The first quantity we are able to derive is a Riemann matrix $\tau$, using the same method as~\cite[\S~3.3]{Rit10}. Starting with the positive definite unimodular Hermitian matrix $M$ corresponding to the polarization $\varphi = \varphi_0 \cdot M$, we obtain the Riemann matrix $\tau$ associated to $\varphi$ and the CM-elliptic curve $E \simeq \CC/(\ZZ+\ZZ \omega)$ where~\mbox{$\omega=(1+\sqrt{\Delta})/2$} if $\Delta$ is odd and $\omega=\sqrt{\Delta}$ otherwise. 

This matrix we get is defined up to the action of the symplectic group $\Sp_4(\ZZ)$. One then works out a matrix $\tau_0$ in the orbit of $\tau$ for which the computation of the theta constants $(\theta_i )_{0 \leq i \leq 9}$ at $\tau_0$ is fast (see~\cite{Lab16} for instance).

A complex model of a curve $C : y^2 = x(x-1)(x-\lambda_1)(x-\lambda_2)(x-\lambda_3)$ with Riemann matrix $\tau_0$ can then be classically approximated using Rosenhain's formulas~\cite[p.~417]{Ros51}
\[\lambda_1 = \frac{\theta_0^2\theta_2^2}{\theta_1^2\theta_3^2}, \qquad \lambda_2 = \frac{\theta_2^2\theta_7^2}{\theta_3^2\theta_9^2}, \qquad \text{and} \qquad \lambda_3 = \frac{\theta_0^2\theta_7^2}{\theta_1^2\theta_9^2}.\]

By computing the theta constants to higher and higher precision, we are able to get a sufficiently good approximation of the Cardona--Quer invariants to recognize them as rationals. The numbers we get are \emph{a priori} only heuristic as there is no bound known for the denominators of these rationals; however, we can sometimes prove that these heuristic values are correct, as follows.

Given a set of Cardona--Quer invariants that we suspect are equal to the invariants of a curve whose Jacobian is isomorphic to $E^2$ for an $E$ with complex multiplication, we can easily produce a curve $C$ having those invariants. Then we can use the techniques of~\cite{CMSV17} to provably compute the endomorphism ring of the Jacobian of $C$. If this endomorphism ring is isomorphic to the ring $M_2(\End E)$, then we have provably found a curve of the type we are looking for.

We computed heuristic values for the Cardona--Quer invariants of our $46$ principally polarized abelian surfaces, and the list of these invariants is available on authors' web pages, together with all the programs to compute them. We are grateful to J.~Sijsling for computing the endomorphism rings for the Jacobians of~$13$ of our $46$ curves; he is currently developing a faster and more robust algorithm which should be able to handle the remaining cases. For each of these $13$ curves, the endomorphism ring was $M_2(\End E)$, so the heuristic values of the Cardona--Quer invariants of these curves are provably correct.

We observe for that the $13$ provably-correct sets of invariants, all the denominators are smooth integers. It would be very interesting, in the same spirit as~\cite{GL07,LV15} for the CM genus-$2$ case, to find formulas to explain the prime powers dividing these denominators. An example of such a closed formula appears in the introduction of~\cite{Rod00} without any details. The denominators of the $33$ sets of invariants that we have not proven to be correct also are smooth, which provides some further heuristic evidence that the values are correct.

We present in Table~\ref{tab:curves} the invariants for a few of the curves we could provably compute.

{\small
\begin{table}[ht]
\centering
\begin{tabular}{r@{\qquad}cc}
\toprule
 $\Delta$ & $M$ & Cardona--Quer invariants $[g_1,g_2,g_3]$ \\
 \midrule
$ -8$ & $\begin{pmatrix} 2 & \omega + 1 \\ -\omega + 1 & 2 \end{pmatrix}$ & $\left[2^4\cdot 5^5,2\cdot 3\cdot 5^4,-5^3\right]$ \\[3ex]
$-11$ & $\begin{pmatrix} 2 & \omega     \\ -\omega + 1 & 2 \end{pmatrix}$ & $\displaystyle\left[\frac{19^5}{2^2},\frac{3^2\cdot 11\cdot 19^3}{2^5},-\frac{19^2\cdot 47}{2^6}\right]$ \\[3ex]
$-19$ & $\begin{pmatrix} 2 & \omega     \\ -\omega + 1 & 3 \end{pmatrix}$ & $\displaystyle\left[\frac{5^5\cdot 29^5}{2^2\cdot 3^7},\frac{5^3\cdot 7\cdot 29^3\cdot 31\cdot 73}{2^5\cdot 3^8},-\frac{5^2\cdot 17\cdot 29^2\cdot 2719}{2^6\cdot 3^{10}}\right]$ \\[3ex]
$-20$ & $\begin{pmatrix} 2 & \omega     \\ -\omega     & 3 \end{pmatrix}$ & $\displaystyle\left[\frac{5^5\cdot 7^5}{2^2},\frac{5^5\cdot 7^3\cdot 11}{2^5},-\frac{3\cdot 5^3\cdot 7^2}{2^6}\right]$ \\[3ex]
$-24$ & $\begin{pmatrix} 2 & \omega + 1 \\ -\omega + 1 & 4 \end{pmatrix}$ & $\displaystyle\left[\frac{2^4\cdot 23^5}{3},\frac{2\cdot 23^3\cdot 421}{3^2},-\frac{23^2\cdot 37}{3^4}\right]$ \\[3ex]
$-40$ & $\begin{pmatrix} 2 & \omega + 1 \\ -\omega + 1 & 6 \end{pmatrix}$ & $\displaystyle\left[\frac{2^4\cdot 5^5\cdot 43^5}{3^7},\frac{2\cdot 5^4\cdot 43^3\cdot 6977}{3^8},-\frac{5^4\cdot 13\cdot 43^2}{3^{10}}\right]$ \\[3ex]
$-52$ & $\begin{pmatrix} 2 & \omega     \\ -\omega     & 7 \end{pmatrix}$ & $\displaystyle\left[\frac{5^5\cdot 173^5}{2^2\cdot 3^7},\frac{5^4\cdot 173^3\cdot 112061}{2^5\cdot 3^ 8},-\frac{5^3\cdot 7\cdot 37\cdot 173^2}{2^6\cdot 3^{10}}\right]$ \\[3ex]
\bottomrule
\end{tabular}
\bigskip
\caption{Cardona--Quer invariants for seven of the $46$ genus-$2$ curves with field of moduli $\QQ$ whose Jacobians are isomorphic to~$E^2$, where $E$ has CM by a maximal order~$\Ocal$. The discriminant of $\Ocal$ is $\Delta$, the corresponding principal polarization on $E^2$ is~$\varphi_0\cdot M$, and $\omega$ denotes either $\sqrt{\Delta}/2$ or $(1 + \sqrt{\Delta})/2$, depending on whether $\Delta$ is even or odd.}
\label{tab:curves}
\end{table}}

\subsection{When is \texorpdfstring{$\QQ$}{Q} also a field of definition for \texorpdfstring{$C$}{C}?}
 
To conclude let us consider any of the $46$ previous pairs $(A,\varphi)$. We know that there exists a genus-$2$ curve $C/\QQbar$ with field of moduli $\QQ$ such that $(\Jac(C),j) \simeq_{\QQbar} (A,\varphi)$, where $j$ is the canonical polarization on $\Jac(C)$. If the order of $\Aut(A,\varphi) \simeq \Aut(C)$ is larger than~$2$, then it is known~\cite{CQ05} that the field of moduli of $C$ is a field of definition and that there exists a genus-$2$ curve $C_0\colon y^2=f(x)$ with $f \in \QQ[x]$ such that \mbox{$(\Jac(C_0),j_0) \simeq_{\QQbar} (A,\varphi)$}. In particular $\QQ$ is also a field of definition for $(A,\varphi)$.

\begin{proposition}[{Compare to~\cite[\S~4]{Rod00}}] \label{prop:fielddef}
The field $\QQ$ is a field of definition of~$C$ --- and therefore of $(A,\varphi)$~--- if and only if
the order of \mbox{$\Aut(A,\varphi)\simeq \Aut(C)$} is larger than~$2$.
\end{proposition}

\begin{proof}
It remains to prove that when $\Aut(A,\varphi) = \{\pm 1 \}$, there is no model of~$(A,\varphi)$ over $\QQ$. Actually we show there is even no model $(B,\mu)$ over $\RR$. Indeed, an isomorphism $\psi : (A,\varphi)/\CC \to (B,\mu)/\RR$, defined over~$\CC$, would induce an isomorphism \[\alpha_{\iota} = (\psi^{-1})^{\iota} \circ \psi : (A,\varphi) \to (A,\varphi)^{\iota},\] for the complex conjugation $\iota$, such that ${\alpha_{\iota}}^{\iota} \circ \alpha_{\iota} = \left((\psi^{-1}) \circ \psi^{\iota}\right) \circ \left((\psi^{-1})^{\iota} \circ \psi\right) = \Id$.

Since we have seen that $E^{\iota} = E$, the isomorphism $\alpha_{\iota}$ can be represented as a matrix $P \in \GL_2(\Ocal)$ such that $\bar{P} P = \Id$. Moreover the commutativity of the diagram 
\[\xymatrix{E^2 \ar[rr]^{\varphi} \ar[d]_{\alpha_{\iota}} && \widehat{E}^2 \\E^2 \ar[rr]^{\varphi^\iota} && \widehat{E}^2\ar[u]_{ \widehat{\alpha}_{\iota}}.\\}\]
translates into the equality $P^* \bar{M} P = M$. If we denote $M=\left(\begin{smallmatrix} a & b \\ \bar{b} & d \end{smallmatrix}\right)$, then it is easy to see that the matrix $P_0= \big(\begin{smallmatrix} \bar{b} & d \\ -a & -b \end{smallmatrix}\big)$ satisfies the last equality. Any other $P=P_0 R$ differs from $P_0$ by an automorphism $R$ of $(A,\varphi)$ since \mbox{$R^* P^* \bar{M} P R= R^* M R = M$}.
Because the automorphism group of $(A,\varphi)$ is $\{\pm 1\}$, this means that the only possible $P$ are $\pm P_0$. It is easy to check that $P_0\bar{P_0} = (-P_0)(-\bar{P_0}) = -\Id$, so the condition~$\bar{P} P = \Id$ cannot be satisfied.
\end{proof}

\subsection{Provably correct equations for the curves defined over \texorpdfstring{$\QQ$}{Q}} 

Using Proposition~\ref{prop:fielddef} we found that exactly $13$ of our curves can be defined over~$\QQ$, and these~$13$ are precisely the curves for which we could provably compute the invariants. This is no coincidence, as having an equation over $\QQ$ definitely simplifies the computation. We present these curves in~Table~\ref{tab:resQ}.

{\small
\begin{table}[ht]
\centering
\begin{tabular}{cc@{\ }c@{\ }l}
\toprule
     $\Delta$ & $M$ &$d$ &\quad Equation for $C$\\
\midrule      
  $-8$ & $\begin{pmatrix} 2 & \omega + 1 \\ -\omega + 1 &  2 \end{pmatrix}$ & $1$            & \begin{tabular}{l}{$y^2 = x^5 + x$}\end{tabular}\\[3ex]
 $-11$ & $\begin{pmatrix} 2 & \omega     \\ -\omega + 1 &  2 \end{pmatrix}$ & $(-11)^{1/3}$  & \begin{tabular}{l}{$y^2 = 2 x^6 + 11 x^3 - 2\cdot11$}\end{tabular}\\[3ex]
 $-19$ & $\begin{pmatrix} 2 & \omega     \\ -\omega + 1 &  3 \end{pmatrix}$ & $-19$          & \begin{tabular}{l}{$y^2 = x^6 + 1026 x^5 + 627 x^4 + 38988 x^3$}\\
                                                                                                                 {$\phantom{y^2 = x^6} - 627\cdot19 x^2 + 1026\cdot19^2 x - 19^3$}\end{tabular}\\[3ex]
 $-43$ & $\begin{pmatrix} 2 & \omega     \\ -\omega + 1 &  6 \end{pmatrix}$ & $-43$          & \begin{tabular}{l}{$y^2 = x^6 + 48762 x^5 + 1419 x^4 + 4193532 x^3$}\\
                                                                                                                 {$\phantom{y^2 = x^6} - 1419\cdot43 x^2 + 48762\cdot43^2 x - 43^3$}\end{tabular}\\[3ex]
 $-67$ & $\begin{pmatrix} 2 & \omega     \\ -\omega + 1 &  9 \end{pmatrix}$ & $-67$          & \begin{tabular}{l}{$y^2 = x^6 + 785106 x^5 + 2211 x^4 + 105204204 x^3$}\\
                                                                                                                 {$\phantom{y^2 = x^6} - 2211\cdot67 x^2 + 785106\cdot67^2 x - 67^3$}\end{tabular}\\[3ex]
$-163$ & $\begin{pmatrix} 2 & \omega     \\ -\omega + 1 & 21 \end{pmatrix}$ & $-163$         & \begin{tabular}{l}{$y^2 = x^6 + 1635420402 x^5 + 5379 x^4$}\\
                                                                                                                 {$\phantom{y^2 = x^6} + 533147051052 x^3 - 5379\cdot163 x^2$}\\
                                                                                                                 {$\phantom{y^2 = x^6} + 1635420402\cdot163^2 x - 163^3$}\end{tabular}\\[3ex]
 $-20$ & $\begin{pmatrix} 2 & \omega     \\ -\omega     &  3 \end{pmatrix}$ & $\sqrt{5}$     & \begin{tabular}{l}{$y^2 = x^5 + 5 x^3 + 5 x$}\end{tabular}\\[3ex]
 $-24$ & $\begin{pmatrix} 2 & \omega + 1 \\ -\omega + 1 &  4 \end{pmatrix}$ & $\sqrt{2}$     & \begin{tabular}{l}{$y^2 = 3 x^5 + 8 x^3 + 3\cdot2 x$}\end{tabular}\\[3ex]
 $-40$ & $\begin{pmatrix} 2 & \omega + 1 \\ -\omega + 1 &  6 \end{pmatrix}$ & $\sqrt{5}$     & \begin{tabular}{l}{$y^2 = 9 x^5 + 40 x^3 + 9\cdot5 x$}\end{tabular}\\[3ex]
 $-52$ & $\begin{pmatrix} 2 & \omega     \\ -\omega     &  7 \end{pmatrix}$ & $\sqrt{13}$    & \begin{tabular}{l}{$y^2 = 9 x^5 + 65 x^3 + 9\cdot13 x$}\end{tabular}\\[3ex]
 $-88$ & $\begin{pmatrix} 2 & \omega + 1 \\ -\omega + 1 & 12 \end{pmatrix}$ & $\sqrt{2}$     & \begin{tabular}{l}{$y^2 = 99 x^5 + 280 x^3 + 99\cdot2 x$}\end{tabular}\\[3ex]
$-148$ & $\begin{pmatrix} 2 & \omega     \\ -\omega     & 19 \end{pmatrix}$ & $\sqrt{37}$    & \begin{tabular}{l}{$y^2 = 441 x^5 + 5365 x^3 + 441\cdot37 x$}\end{tabular}\\[3ex]
$-232$ & $\begin{pmatrix} 2 & \omega + 1 \\ -\omega + 1 & 30 \end{pmatrix}$ & $\sqrt{29}$    & \begin{tabular}{l}{$y^2 = 9801 x^5 + 105560 x^3 + 9801\cdot29 x$}\end{tabular}\\
\bottomrule
\end{tabular}
\bigskip
\caption{Genus-$2$ curves defined over $\QQ$ with Jacobian isomorphic over $\QQbar$ to $E^2$, where $E$ has CM by a maximal order $\Ocal$. The discriminant of $\Ocal$ is $\Delta$, the corresponding principal polarization on~$E^2$ is $\varphi_0\cdot M$, and $\omega$ denotes either $\sqrt{\Delta}/2$ or $(1 + \sqrt{\Delta})/2$, depending on whether $\Delta$ is even or odd. This list is complete if the Generalized Riemann Hypothesis holds. Each curve is a double cover of its corresponding $E$ (as can be seen by the fact that the upper-left entry of each polarization matrix is~$2$), and the associated involution of $C$ is given by $(x,y)\mapsto (d/x, d^{3/2} y/x^3)$ for the value of $d$ given in the third column.}
\label{tab:resQ}
\end{table}}

\bibliographystyle{alphadoi}
\bibliography{GelinEtAl}
\end{document}